\documentclass[11pt]{amsart}
\usepackage{geometry}                % See geometry.pdf to learn the layout options. There are lots.
\usepackage{graphicx}
\usepackage{amssymb}
\usepackage{epstopdf}
\theoremstyle{plain}
\newtheorem{theo}{Theorem}[section]
\newtheorem{lem}[theo]{Lemma}

\newtheorem{prop}[theo]{Proposition}

\numberwithin{equation}{section}

\theoremstyle{definition}

\newtheorem{remark}[theo]{Remark}

\def\P{\mathfrak P}

\def\D{\Delta}

\def\g{\gamma}

\def\l{\lambda}

\def \R{\mathbb R}

\def \P{\mathbb P}

\def\Q{\mathbb Q}
\def\W{\mathcal W}
\def\wh{\widehat}
\def\wt{\widetilde}
\def\({\biggl(}
\def\){\biggr)}

\def\<{\bold\langle}
\def\>{\bold\rangle}

\def\M{\widetilde {M}}
\def\MM{\widetilde {M}(\infty)}

\def\dd{\widetilde{\delta}}

\begin{document}

\title{Entropy rigidity of symmetric spaces without focal points}
\author{Fran\c cois Ledrappier  and  Lin Shu}
\address{Fran\c cois Ledrappier,  Department of Mathematics, University of Notre Dame, IN 46556-4618, USA}\email{fledrapp@nd.edu}
\address{Lin Shu, LMAM, School
of Mathematical Sciences, Peking University, Beijing 100871,
People's Republic of China} \email{lshu@math.pku.edu.cn}

%\date{}                                            %Activate to display a given date or no date

\maketitle \centerline{\it{To Werner Ballmann for his 61st
birthday}}
\begin{abstract}
We characterize symmetric spaces without focal points by the
equality case of general inequalities between geometric quantities.
\end{abstract}
\section{Introduction}
%\subsection{}

Let $(M,g)$ be a closed connected Riemannian manifold, and $\pi :
(\M, \wt g) \to (M,g)$ its universal cover endowed with the lifted
Riemannian metric. We denote $p(t,x,y), t \in \R_+, x,y \in \M $ the
heat kernel on $\M$, the fundamental solution of the heat equation $
\frac{\partial u }{\partial t} = \Delta u $ on $\M$, where
$\Delta={\textrm { Div }}\nabla$ is the Laplacian on $\M$. Since we
have a compact quotient, all the following limits exist as $t \to
\infty $ and are independent of $x \in \M$:
\begin{eqnarray*}
\l_0 \;&=&\;  \inf _{f \in C^2_c(\M)} \frac {\int |\nabla f|^2}{\int |f|^2} \; = \; \lim _t - \frac{1}{t} \ln p(t,x,x) \\
\ell \; &=&\;  \lim _t \frac{1}{t} \int d(x,y) p(t,x,y) \ d{\textrm {Vol}}(y) \\
h \; &=&\; \lim_t - \frac{1}{t} \int p(t,x,y) \ln p(t,x,y) \ d{\textrm {Vol}}(y) \\
v\; &=& \; \lim _t \frac{1}{t} \ln {\textrm {Vol}} B_{\M} (x,t) ,
\end{eqnarray*}
where $B_{\M} (x,t) $ is the ball of radius $t$ centered at $x$ in
$\M$ and ${\textrm {Vol}}$ is the Riemannian volume on $\M$.

All these numbers are nonnegative. Recall $\l_0$ is the bottom of
the spectrum of the Laplacian, $\ell $ the linear drift, $h$ the
stochastic entropy and $v$ the volume entropy. There is the
following relation:
\begin{equation}\label{ineq}
4\l_0 \; \stackrel {(a)}{\leq} \; h \;\stackrel {(b)}{ \leq }\; \ell
v \;\stackrel{(c)}{ \leq }v^2.
\end{equation}
See \cite{L1} for (a), \cite{Gu} for (b). Inequality (c) is shown in
\cite{L4} as a corollary of (b) and (\ref{basic}):
\begin{equation}\label{basic}
\ell ^2 \; \leq \; h.
\end{equation}

In this paper,  we are interested in the characterization of locally
symmetric property of $M$ by the  equality case of  the above
inequalities for manifolds without focal points.  Recall that a
Riemannian manifold $M$ is said to have \emph{no focal points} if
for any imbedded open geodesic segment $\gamma: (-a, a)\mapsto M$
(where $0<a\leq \infty$), the restriction of the exponential map  on
the normal bundle of $\gamma$  is everywhere nonsingular.  Hence any
manifold of nonpositive curvature has no focal points.  The reverse
is not true since there exist manifolds without focal points but
with sectional curvatures of both signs (\cite{Gul}).  If $M$ is a
locally symmetric space without focal points,  it must have
nonpositive curvature (\cite[Theorem 3.1]{He}). Note that for
symmetric spaces of nonpositive curvature, $4\lambda_0=v^2$
(\cite{Kar}, cf. \cite[Appendice C]{BeCG91}) and hence all five
numbers $4\l_0, \ell ^2, h, \ell v, v^2$ coincide by (\ref{ineq}),
(\ref{basic}) above and are positive unless $(\M, g )$ is $(\R^n,
{\textrm {Eucl.}})$.  Our result is a partial converse:

\begin{theo}\label{main}Let $(M, g)$ be a compact connected Riemannian manifold without focal points. With  the above notation,  all the following equalities are equivalent to the locally symmetric property of $M$:
\begin{itemize}
\item[i)] $4\lambda_0=  v^2$;
\item[ii)] $h=\ell^2\ \mbox{or}\  v^2$;
\item[iii)] $\ell=v$.
\end{itemize}
\end{theo}

As recalled in \cite{L4}, Theorem \ref{main}  is known in negative
curvature case and follows from \cite{K1}, \cite{BFL}, \cite{BCG95},
\cite{FL}  and  \cite{L1}.  The other possible converses are
delicate: even for  negatively  curved manifolds, in dimension
greater than two, it is not known whether $h=\ell v$ holds only for
locally symmetric spaces. This is equivalent to a conjecture of
Sullivan (see \cite{L3} for a discussion). Sullivan's conjecture
holds for surfaces of negative curvature (\cite {L1}, \cite {Ka}).
In negative curvature case,   $4\lambda_0=h$ (and hence $4\lambda_0=
\ell v$) implies $M$ is locally symmetric, which follows from
\cite{L2},  \cite{BFL}, \cite{BCG95},  and \cite{FL}. In the no
focal points case, whether $4\lambda_0=h$ holds only for locally
symmetric spaces may depend on a further study on the Martin
boundary of $\M$ (cf. \cite{AS} for this notion). However, it would
follow from the hypothetical $4\lambda_0\leq \ell^2$ by ii) of
Theorem \ref{main}.

We assume henceforth that $(M,g)$ has no focal points.  Given a
geodesic $\g $ in $M$, Jacobi fields along $\g$ are vector fields $
t \mapsto J(t) \in T_{\g(t)}M$ which describe infinitesimal
variation of geodesics around $\g$. The rank of the geodesic $\g$ is
the dimension of the space of Jacobi fields such that $t \mapsto
\|J(t)\| $ is a constant function on $\R$. The rank of a geodesic
$\g$ is at least one because of the trivial $t \mapsto \dot {\g}(t)$
which describes the variation by sliding the geodesic along itself.
The rank of the manifold $M$ is the smallest rank of geodesics in
$M$. Using  rank rigidity theorem for manifolds with no focal points
from \cite{Wa}, we reduce in section 2 the proof of Theorem
\ref{main} to proving that if $(M,g)$ is rank one, the equality
$\ell^2=h$  implies that $(\M, \wt g) $ is a symmetric space.  For
this, we show in section 4 that the equality $\ell^2=h$  implies
that $(\M, \wt g) $ is asymptotically harmonic (see section
\ref{Proof} for its definition). This uses the solvability of the
Dirichlet problem at infinity and the structure of harmonic measures
of the stable foliation, which will be presented in section 3.
Finally, it was recently observed by A. Zimmer (\cite{Zi1}) that
asymptotically harmonic universal covers of rank one manifolds are
indeed symmetric spaces.

\section{Generalities and reduction of Theorem \ref{main}}

We assume in the following that $(M, g)$ is a compact connected
Riemannian manifold without focal points.  Let $\wt{M}$ be the
universal cover of $M$ with covering group $\Gamma=\pi_1(M)$.

\subsection{Geometric boundary}%\label{Geo.boundary}
The notion of geometric boundary was introduced by Eberlein and
O'Neill \cite{EO} for nonpositive curvature case (see \cite{B3}).
Consider the geodesics on $\wt{M}$.  Two geodesic rays $\gamma_1$
and $\gamma_2$ of $\wt{M}$ are said to be asymptotic (or
equivalent) if $\sup_{t\geq 0}d(\gamma_1(t), \gamma_2(t))<+\infty$.
The set of equivalence classes $[\gamma]$, denoted by
$\wt{M}(\infty)$, is called the geometric boundary of $\wt{M}$.  We
denote by $\wh{M}$ the union $\M\cup \M(\infty)$.

Given $x\in \M$, for any geodesic ray $\gamma$ of $\M$, there exists
a unique geodesic  starting  at $x$ which is asymptotic to $\gamma$
(\cite[Proposition 3]{O}).  Hence for any $(x, \xi)\in \M\times
\M(\infty)$, there is a unique unit speed geodesic $\gamma_{x, \xi}$
satisfying $\gamma_{x, \xi}(0)=x$ and $[\gamma_{x, \xi}]=\xi$.
Denote by $S_x\M$  the unit tangent space at $x$ to $\M$. The
mapping   $\pi_x ^{-1}:\  \MM\mapsto S_x\M$ sending  $\xi$ to
$\dot\g_{x, \xi}(0)$ is a bijection between these two sets.

For  $v, w\in S_p \M$, $p\in \M$,  the angle $\angle_p (v, w)$ is
the unique number $0\leq \theta\leq \pi$ such that  $\langle v,
w\rangle=\cos  \theta$. For $q\in \wh{M}$ other than $p$, let
$\g_{p, q}$ denote the unique unit speed geodesic starting at $p$
pointing at $q$.   Given  $v\in S_p \M$ and  $0<\epsilon<\pi$.  The
set
\[C(v, \epsilon):=\{q\in \wh{M}:\ \ \angle_p (v, \dot\g_{p, q}(0))<\epsilon\}\]
 is called  the cone of vertex $p$, axis $v$, and angle $\epsilon$ (cf. \cite{EO}).   It was shown in \cite{Go} that there exists a  canonical topology  on  $\wh{M}$  so that  for any $x\in \M$,  the mapping  $\pi_x $ is a homeomorphism   between  $S_x\M$ and  $\MM$.   The topology is called  the ``cone''  topology  in the sense for $\xi\in \M (\infty)$,
 the truncated cones
 \[C(v,  \epsilon, r)=C(v, \epsilon)\cap \left(\widehat{M}\backslash B(p, r)\right)\]
 containing  $\xi$,  where $B(p, r)$ is the closed ball of radius $r$ about $p$, form a local basis at $\xi$.

 We will identify $S\M$ with $\M \times \MM$ by $(x,v) \mapsto (x, \pi _x v)$.  The action of  $\Gamma$ on $M$ can be continuously extended to $\wt{M}(\infty)$.  Hence the quotient $SM$ is identified with the quotient of $\M \times \MM$ under the diagonal action of $\Gamma$.

\

\subsection{Stable Jacobi tensor}%\label{stable tensor}

Let $\g$ be a geodesic in $(\M,\wt{g})$ and let $N(\g) $ be the
normal bundle of $\g$: $$ N(\g ) :=\cup _{t \in \R}N_t (\g),
{\textrm { where }} N_t(\g) \; = \; ( \dot \g (t))^\perp \; = \; \{
X \in T_{\g(t)} M:\ \ \< X, \dot \g (t)\> = 0 \}.$$ A (1,1) tensor
along $\g$ is a family $V = \{V(t), t\in \R\}$, where $V(t) $ is an
endomorphism of $N_t(\gamma)$ such that for any family $Y_t$ of
parallel vectors along $\g$, the covariant derivative $V'(t) Y_t :=
\frac {D}{dt} V(t) Y_t $ exists.

We endow $N(\g)$ with Fermi orthonormal coordinates  given by a
parallel frame field along $\g $. A (1,1) tensor along $\g$ is
parallel if $V' (t) = 0 $ for all $t$. It is then given by a
constant matrix  in Fermi coordinates. The curvature tensor $R$
induces a symmetric (1,1) tensor along $\g$ by  $R(t)X = R(X, \dot
\g (t)) \dot\g (t).$ A (1,1) tensor $V(t)$ along $\g$ is called a
\emph{Jacobi tensor} if it satisfies $ V'' + R V = 0 .$ If $V(t)$ is
a Jacobi tensor along $\g$, then $J(t) := V(t) Y_t$ is a Jacobi
field for any parallel field $Y_t$.

For each ${\bf{v}}=(x, v)\in S\M$, let $\gamma_{\bf{v}}$ denote the
unique geodesic starting from $x$ with speed $v$. A Jacobi tensor
$V_{{\bf v}}$ defined for each $\gamma_{\bf v}$ is called continuous
if the initial values $V_{\bf v}(0), V'_{\bf v}(0)$ are continuous
as $(1,1)$ tensors of the vector bundle  ${\bf B}:=\{({\bf v}, {\bf
w})\in S\M\times T\M:\ \ {\bf w}\perp {\bf v}\}$  over $S\M$.

Now for ${\bf{v}}=(x, v)\in S\M$,  denote by $A_{\bf{v}}$ be the
Jacobi tensor along $\gamma_{\bf{v}}$ with initial condition
$A_{\bf{v}}(0)=0$ and $A'_{\bf{v}}(0)=I$ ($I$ is the identity).  For
each $s>0$, let $S_{{\bf{v}}, s}$ be the Jacobi tensor with the
boundary conditions $S_{{\bf{v}}, s}(0)=I$ and $S_{{\bf{v}},
s}(s)=0$.  It can be shown (cf. \cite{EOs}) that the limit
$\lim_{s\to +\infty}S_{{{\bf v}}, s}=: S_{{\bf v}}$ exists and is
given by
\[S_{{\bf v}}(t)=A_{{\bf v}}\int_{t}^{+\infty} (A_{{\bf v}}^{*}A_{{\bf v}})^{-1}(u)\ du,\]
where $A_{{\bf v}}^*$ is the transposed form of $A_{{\bf v}}$. The
tensor $S_{{\bf v}}$ is called the stable tensor along the geodesic
$\g_{\bf v}$. As a consequence of the  uniform convergence of
$S'_{{\bf{v}}, s}(0)$ to $S'_{{\bf v}}(0)$ (\cite[Proposition
5]{E}),  one has by continuity of $S'_{{\bf{v}}, s}(0)$ (with
respect to ${\bf v}$) that the stable tensor $S_{\bf v}$ is
continuous with respect to ${\bf v}$ (\cite[Proposition 4]{EOs}).

For each ${\bf{v}}=(x, v)\in S\M$, the vectors $(Y, S_{\bf v}(0)Y)$
describe variations of asymptotic geodesics and the subspace $E_{\bf
v}^s\subset T_{\bf v}T\M$ they generate corresponds to $TW_{\bf
v}^{s}$, where $W_{\bf v}^{s}$, the set of initial vectors of
geodesics asymptotic to $\g_{\bf v}$, is identified with $\M\times
\pi_x(v)$ in $\M\times \M(\infty)$.  Recall that $SM $ is identified
with the quotient of $\M \times \MM$ under the diagonal action of
$\Gamma$. Clearly, for $\varphi \in \Gamma$, $\varphi(W^s_{\bf v}) =
W^s_{D\varphi {\bf v}}$ so that the collection of $W^s_{\bf v}$
define a foliation $\W^s$ on $SM$, the so-called \emph{stable
foliation} of $SM$. The leaves of the stable foliation $\W^s$ are
quotient of $\M$, they are naturally endowed with the Riemannian
metric induced from $\wt g$.

Similarly, by reversing the time in the  construction of stable
tensor,  one obtains the corresponding unstable tensor and hence the
unstable subspaces and  the unstable foliation.

\

\subsection{Busemann functions}\label{sec-Busemann}
Fix $x_0\in \M$ as a reference point.  For each $\xi\in \M(\infty)$,
define a Busemann function at $\xi$ (cf. \cite{E}) as follows. Let
$\gamma_{x_0, \xi}$ be the unique unit speed geodesic starting at
$x_0$ which is asymptotic to $\xi$.  For each $s\geq 0$, define the
function
\[
b_{\xi, s}(x):=d(x, \gamma_{x_0, \xi}(s))-s, \ \forall  x\in \M.
\]
We have by triangle inequality  that  $b_{\xi, s}(x)$ are decreasing
with $s$ and  bounded absolutely from below by $-d(x_0, x)$. So the
function
\[
b_{\xi} (x):=\lim\limits_{s\rightarrow \infty} b_{\xi, s}(x), \
\forall  x\in \M,
\]
is well defined and is called the Busemann function at $\xi$. It was
shown in \cite{E} that the function $x\mapsto b_{\xi}(x)$ is of
class $C^2$.

For each $\varphi\in \Gamma$, $(x, \xi)\in \M\times \M(\infty)$, we
consider $b_{\varphi\xi}(\varphi x)$. The geodesics $\gamma_{x_0,
\varphi \xi}$ and $\g_{\varphi x_0, \varphi\xi}$ are asymptotic.
Hence the Busemann functions at $\varphi \xi$ using different
reference points $x_0$ and $\varphi x_0$ only differ by a constant
depending on $x_0, \varphi, \xi$ (\cite[Proposition 3]{E}), which is
in fact  given by $b_{ \xi}(\varphi x_0)$.  So we have
\begin{eqnarray*}
b_{\varphi\xi}(\varphi x)&:=& \lim\limits_{s\to +\infty}\left(d(\varphi x, \g_{x_0,  \varphi \xi}(s))-s\right)\\
&=&   \lim\limits_{s\to +\infty}\left(d(\varphi x, \g_{\varphi x_0,  \varphi \xi}(s))-s\right)+b_{\xi}(\varphi x_0)\\
&=& \lim\limits_{s\to +\infty} (d(x, \gamma_{x_0, \xi}(s))-s)+b_{\xi}(\varphi x_0)\\
&=& b_{\xi}(x)+b_{\xi}(\varphi x_0).
\end{eqnarray*}
It follows that the function $\Delta_x b_{\xi}$ satisfies
$\Delta_{\varphi x} b_{\varphi \xi}=\Delta_{x} b_{\xi}$ and
therefore defines a function $B$ on  $\Gamma\backslash (\M\times
\M(\infty))=SM$, which is called the Laplacian of the Busemann
function.

For ${\bf v}=(x, v)\in S\M$, let $\xi=[\gamma_{{\bf v}}]$ and let
$b_{\bf v}:=b_{x, \xi}$, where $b_{x, \xi}$ is the Busemann function
at $\xi$ using $x$ as a reference point.  It is true (\cite{E}) that
\begin{equation*}%\label{Bus-stable}
\nabla_{w}(\nabla b_{\bf v})=-S'_{\bf v}(0) (w).
\end{equation*}
Since the stable tensor $S_{\bf v}$ is continuous with respect to
${\bf v}$,  we have $\Delta_x b_{\bf v}=-\mbox{Tr}S'_{\bf v}(0)$
also depends continuously on $v$.  Note that $\Delta_x b_{\bf
v}=\Delta_x b_{\xi}$,  we have $\Delta_x b_{\xi}$ depends
continuously on $\xi$.  Consequently the function $B$ is continuous
on $SM$.

\

\subsection{Proof of Theorem \ref{main}}\label{Proof}
We continue assuming that $(\M,\wt g)$ has  no focal points. By the
Rank Rigidity Theorem (see \cite {Wa}), $(\M ,\wt g) $ is of the
form $$( \M_0 \times \M_1 \times \cdots \times \M_j \times \M_{j+1}
\times \cdots \times \M_k , \wt g)\footnote{With a clear convention
for the cases  when Dim $\M_0 =0,$  $j =0 $ or   $k=j$.} ,$$ where
$\wt g$ is the product metric $\wt g^2 = (\wt g_0)^2 + (\wt g_1)^2 +
\cdots +  (\wt g_j)^2 +  (\wt g_{j+1})^2 +\cdots +  (\wt g_k)^2 $,
$(\M_0 , \wt g_0) $ is Euclidean, $(\M_i, \wt g_i) $ is an
irreducible symmetric space of rank at least two for $i = 1, \cdots,
j$ and a rank one manifold  for $i = j+1, \cdots , k.$ If the
$(\M_i, \wt g_i), i = j+1, \cdots, k, $ are all symmetric spaces of
rank one, then $(\M, \wt g)$ is a symmetric space. Moreover in that
case, all inequalities in (\ref {ineq}) are equalities: this is the
case for irreducible symmetric spaces (all numbers are 0 for
Euclidean space;  for the other spaces, note that locally a
symmetric space without focal points must have nonpositive curvature
(\cite[Theorem 3.1]{He}) and $4\l_0 $ and $v^2$ are classically
known to coincide for locally symmetric space with nonpositive
curvature (cf. \cite[Appendice C]{BeCG91})) and we have:
$$ 4\l_0 (\M) \; = \; \sum _i 4 \l_0 (\M_i ), \quad v^2 (\M) \; = \; \sum _i v^2 (\M_i).$$
To prove Theorem \ref{main}, it suffices to prove that if $\ell ^2 =
h $ (or  $4\lambda_0=h$), all $\M_i $ in the decomposition are
symmetric spaces. This is already true for $i = 0, 1, \cdots,  j.$
It remains to show that $(\M_i, \wt g_i )$ are symmetric spaces for
$i = j+1, \cdots,  k$. Note that each  one of the spaces $(\M_i, \wt
g_i) $ admits a cocompact discrete group of isometries (see
\cite[Theorem 3.3]{Kn2} using the corresponding theorems from
\cite{Wa}).   This shows that  the linear drifts $\ell _i $ and the
stochastic entropies $h_i $ exist for each one of the spaces $(\M_i,
\wt g_i) $. Moreover, we clearly have
$$ \ell ^2 \; = \; \sum _i \ell _i^2, \quad h\; = \; \sum h_i .$$
Therefore Theorem \ref{main} follows from
\begin{theo} \label{rankone-thm} Assume $(M,g)$ is a closed connected  rank one manifold without focal points and that $\ell ^2 = h$. Then $(\M, \wt g)$ is a symmetric space. \end{theo}

  Let  $(M, g)$ be a closed connected Riemannian manifold without focal points as before.  Its universal cover $(\M, \wt g)$  is said to be  \emph{asymptotically harmonic}  if $B$, the Laplacian of the Busemann function,   is constant on $SM$.   In that case, we  have by \cite[Theorem 1.2]{Zi1} that $(M, g)$ is either flat or the geodesic flow on $SM$ is Anosov.  The latter case, as was observed by Knieper  \cite[Theorem 3.6]{Kn}, actually implies $M$ is a rank one locally symmetric space. (Indeed, let  $M$  be as above with $\M$ being asymptotically harmonic.  If the geodesic flow on $SM$ is Anosov,  then  it is true by P. Foulon and F. Labourie \cite{FL}  that the stable and unstable distribution $E^s$ and $E^u$ of the geodesic flow are $C^{\infty}$. Hence  the result of Y. Benoist, P. Foulon and F. Labourie \cite{BFL} applies and gives that the geodesic flow of $(M, g)$ is smoothly conjugate to the geodesic flow of a locally symmetric space $(M_0, g_0)$ of negative curvature.  Note that in the no focal points case, the volume entropy and the topological entropy of the geodesic flow coincide \cite{FM}.  Thus one can use G. Besson, G. Courtois and S. Gallot's rigidity theorem (\cite{BCG95})  to conclude that the two spaces  $(M, g)$ and $(M_0, g_0)$  are isometric.)   In summary, we have

\begin{prop}\label{final.Zi}(\cite[Theorem 1.1]{Zi1}) Assume $(M, g)$ is a closed connected rank one manifold without focal points such that $(\M, \wt{g})$ is asymptotically harmonic. Then $(\M, \wt{g})$ is a symmetric space.
\end{prop}

Therefore, Theorem \ref{rankone-thm} directly follows from
Proposition \ref{final.Zi} and

\begin{prop}\label{as.har}Assume $(M, g)$ is a closed connected rank one manifold without focal points and that $\ell^2=h$. Then $(\M, \wt{g})$ is asymptotically harmonic.
\end{prop}

\

\section{Harmonic measure for  the stable foliation}
We consider the stable  foliation $\W:=\W^s$  of subsection 2.2.
Recall that the leaves are endowed with a natural Riemannian metric.
We write $\D^\W $ for the associated Laplace  operator on functions
which are of class $C^2$ along the leaves of $\W$. A probability
measure $m$ on $SM$ is called harmonic if it satisfies, for any
$C^2$ function $f$,
$$ \int _{SM} \D^\W f  \ dm \; = \; 0 .$$
Our main result in this section is:

\begin{theo}\label{har.mea}
Let $(M, g)$ be a closed connected rank one manifold without focal
points, $\W$ the stable foliation on $SM$ endowed with the natural
metric as above. Then, there is only one harmonic probability
measure $m$ and the support of $m$ is the whole space $SM$.
\end{theo}

A main step in the proof of  Theorem \ref{har.mea} is to identify
the lift of $m$ in $\M\times \M (\infty)$ locally as $dx\times
dm_x$, where $m_x$ is the hitting probability at $\M(\infty)$ of the
Brownian motion on $\M$ starting at $x$ and  $dx$ is proportional to
the Riemannian volume on $\M$.  We adopt  Ballmann's approach
(\cite{B2,B3}) to use Lyons-Sullivan's procedure (\cite{LS}) of
discretizing the Brownian motion (starting at $x$)  to a random walk
on $\Gamma$ and show $m_x$ is indeed the unique stationary measure
of the corresponding random process on $\Gamma$.   The argument
involves random walk on $\Gamma$,   the solvability of the Dirichlet
problem at infinity  and  divergence properties of geodesics for
manifolds without focal points and is  divided into seven  steps for
clarity.

\

\subsection{Discretization of Brownian motion}%\label{Dis.BM}
Fix $x_0\in \M$. The discretization procedure of Lyons and Sullivan
(\cite{LS}) associates to the Brownian motion on $\M$ a probability
measure $\nu$ on $\Gamma$ with $\nu(\varphi)>0$ for all $\varphi\in
\Gamma$ such that  any bounded harmonic function $h$ on $\M$
satisfies
\[
h(x_0)=\sum\limits_{\varphi\in \Gamma}h(\varphi x_0)\nu(\varphi).
\]
Consider the random walk on $\Gamma$ defined by $\nu$ with
transition probability of $\nu(\varphi^{-1}\wt{\varphi})$ from
$\varphi\in \Gamma$ to $\wt{\varphi}\in \Gamma$.  For given
$\varphi_1, \cdots, \varphi_k$, the probability $\Q$ that a sequence
$\{\varphi_n\}$ begins with $\varphi_1, \cdots, \varphi_k$ is
defined to be
\[
\nu(\varphi_1)\nu(\varphi_1^{-1}\varphi_2)\cdots
\nu(\varphi_{k-1}^{-1}\varphi_{k}).
\]
The  random walk on $\Gamma$ generated by $\nu$ is a good
approximation of Brownian motion on $\M$ starting at $x_0$ in the
following sense:

\begin{prop}\label{Lyons-S}(\cite[Theorem 6]{LS}, cf. \cite[Sec. 4]{B2})  Let $W$ be the space of all continuous paths $c:(0, +\infty)\to \M$ and $\Omega$ the space of all sequences of heads and tails. There is a probability measure $\overline{\P}$ on $W\times \Omega$ with the following properties:
\begin{itemize}
\item[i)] The natural projection from $W\times \Omega$ to $W$ maps $\overline{\P}$ to $\P$, the probability measure on $W$ associated to the Brownian motion starting at $x_0$.
\item[ii)]There is a map $W\times \Omega\to \Gamma^{\Bbb N}$, $(c, \omega)\to \{\varphi_n(c, \omega)\}$, which maps $\overline{\P}$ onto the probability measure $\Q$.
\item[iii)] There is an increasing sequence of stopping times $T_n$ on $W\times \Omega$ and a positive constant $\delta<1$ such that
\[
\overline{\P}\left[ \max\limits_{T_n<t<T_{n+1}} d(c_{\omega}(t),
\varphi_n(c, \omega))>k\right]\leq \delta^k, \ \forall   k>0.
\]
\end{itemize}
\end{prop}

\

\subsection{Stationary measure of the random walk on $\Gamma$}

Let $\nu$ be the Lyons-Sullivan measure on $\Gamma$ corresponding to
$x_0$. Given a probability measure $\mu$ on $\M(\infty)$, define the
convolution $\nu * \mu$ by  letting
\[
\int_{\M(\infty)} f(\xi)\ d(\nu*\mu)(\xi)=\sum_{\varphi\in \Gamma}
\left(\int_{\M(\infty)} f(\varphi\xi) \
d\mu(\xi)\right)\nu(\varphi),
\]
where $f$ is any bounded measurable function on $\M(\infty)$. The
measure $\mu$ is called stationary (with respect to $\nu$) if
\[
\nu * \mu=\mu.
\]
Stationary measures with respect to $\nu$ always exist and are not
supported on points (\cite[p. 56]{B3}). For  the uniqueness of the
harmonic measure for the stable foliation, we first show the hitting
probability of the random walk on $\Gamma$ defined by $\nu$ is the
unique stationary measure on $\M(\infty)$ with respect to $\nu$ and
then identify this measure with the hitting probability of Brownian
motion starting at $x_0$.

\begin{theo}\label{sta.mea}
Let $(M, g)$ be a closed connected rank one manifold without focal
points. Let $\nu$ be a Lyons-Sullivan measure on $\Gamma$ as above.
Then for $\Q$-almost all sequence $\{\varphi_n\}$ in $\Gamma
^\mathbb N$, the sequence $\{\varphi_n x\}$,  $x\in \M$, tends to a
limit in $\M(\infty)$. The hitting probability is given by the
unique stationary measure on $\M(\infty)$ with respect to $\nu$.
\end{theo}

The proof of  Theorem \ref{sta.mea} is the same as in \cite[Theorem
4.11]{B3} once we show the induced random walk on $\Gamma$  is
transient and the Dirichlet problem for $\Gamma$ is solvable (see
section \ref{Dirichlet} for the definition).  In the following, we
present successively  some properties of hyperbolic points for
$\Gamma$,  the transient property of random walk on $\Gamma$, and
finally the solvability of Dirichlet problem for the random walk on
$\Gamma$.

\

\subsection{Hyperbolic points at infinity}
A point $p\in \M(\infty)$ is called \emph{hyperbolic} (\cite{BE}) if
for any $q\not=p$ in $\M(\infty)$, there exists a rank one geodesic
joining $q$ to $p$. A geodesic $\gamma$ of $\M$ is called an axis if
there exists a $\varphi\in \Gamma$ and $a\in \Bbb R$  with $\varphi
(\gamma(t))=\gamma (t+a)$ for all $t$. The endpoints of any rank one
axial geodesic are hyperbolic (\cite[Theorem 6.11]{Wa}).  On the
other hand,  we have by \cite[Proposition 6.12]{Wa} that for any
pair of neighborhoods $U$, $V$ of the endpoints of a rank one
geodesic at $\M (\infty)$, there exists a rank one axis with two
endpoints in $U, V$, respectively.  Note that the geodesic flow is
topologically transitive for rank one manifold without focal points
(\cite{Hu}).  We have by the above argument that

\begin{lem}\label{hy.point} Let $(M, g)$ be a closed connected rank one manifold without focal points. The set of hyperbolic points are dense in $\M(\infty)$.
\end{lem}
For hyperbolic points at infinity, we  also have the following two
lemmas from \cite{Wa}:

\begin{lem}\label{hy.neigh}(\cite[Lemma 6.18]{Wa}) Let $p, q$ be the distinct points in $\M(\infty)$ with $p$ hyperbolic and suppose $U_p$, $U_q$ are neighborhoods of $p$ and $q$, respectively. Then there exists an isometry $\varphi\in \Gamma$ with
\[
\varphi(\wh{M}\backslash U_q)\subset U_p, \ \
\varphi^{-1}(\wh{M}\backslash U_p)\subset U_q.
\]
\end{lem}

\begin{lem}\label{hy.angle}(\cite[Lemma 6.19]{Wa}) Let $p\in \M(\infty)$ be hyperbolic, $U^*\subset \wh{M}$ a neighborhood of $p$, and $x\in \M$. Then there exists a neighborhood $U\subset \wh{M}$ of $p$ such that if $\{\varphi_n\}$ is a sequence of isometries with $\varphi_n (x)\rightarrow x^*\in \M(\infty)\backslash U^*$, then
\[
\sup_{u\in U} \angle_{\varphi_n(x)} (x, u)\to 0, \ \mbox{as}\ \
n\rightarrow \infty,
\]
where $\angle_{a}(b, c)$ denotes the angle between the unit tangent
vectors at $a$ of the geodesics $\gamma_{a, b}$ and $\gamma_{a, c}$.
\end{lem}

\

\subsection{Transient random walk on $\Gamma$}

If $\Gamma$ is not amenable, then  it is true by Furstenberg
\cite[p. 212]{Fu} that the  random walk on $\Gamma$ generated by
$\nu$ is  transient,  i.e.  $d(x, \varphi_n x)\to \infty$ for $\P$
almost any sequence $\{\varphi_n\}\subset \Gamma$ (cf. \cite[p.
58]{B3}).  We show

\begin{lem}  The covering group $\Gamma$ of a closed connected rank one manifold without focal points contains a free subgroup and hence is not amenable.
\end{lem}
\begin{proof} By Lemma \ref{hy.point},  we can choose  two neighborhoods  $U_i$, $i=1, 2$,  of two hyperbolic points  at  $\M(\infty)$ so that
each contains some  additional point besides the hyperbolic point
and they satisfy
\[
U_1\cap U_2=\O, \  U_1\cup U_2\not=\M(\infty).
\]
Then by Lemma  \ref{hy.neigh}, there exist isometries $\varphi_1$
and $\varphi_2\in \Gamma$ with
\[
\varphi_1 (\M(\infty)\backslash U_1)\subset U_1,\ \ \varphi_2
(\M(\infty)\backslash U_2)\subset U_2.
\]
Hence for any $\xi\in \M(\infty)\backslash (U_1\cup  U_2)$ and any
non-trivial word ${\bf \varphi}$ in $\varphi_1$ and $\varphi_2$, we
see that ${\bf \varphi}(\xi)$ belongs to $U_i$ if ${\bf \varphi}$
begins with $\varphi_i$  and so ${\bf \varphi}\not\equiv \mbox{id}$.
Consequently, $\varphi_1$ and $\varphi_2$ generate a free subgroup
of $\Gamma$ and $\Gamma$ is non-amenable as desired.
\end{proof}
\

\subsection{Dirichlet Problem at infinity for $\Gamma$} \label{Dirichlet}

A function $h:\ \Gamma\to \Bbb R$ is called $\nu$ harmonic if
\[
h(\varphi)=\sum_{\psi\in \Gamma} h(\varphi\psi)\nu(\psi),  \
\mbox{for any }\  \varphi\in \Gamma.
\]
Let $\mu $ be a $\nu $-stationary measure. For any bounded and
measurable function $f$ on $\M(\infty)$, define $h_{f}$ on $\Gamma$
by letting
\begin{equation}\label{eq-h_f}
h_{f}(\varphi)=\int_{\M(\infty)} f(\varphi \xi)\ d\mu(\xi).
\end{equation}
Then  $h_{f}$ is $\nu$ harmonic since $\mu$ is $\nu$-stationary. The
Dirichlet problem for random walk on $\Gamma$  generated by  $\nu$
is solvable if for any bounded measurable function $f$ on
$\M(\infty)$ and $\xi\in \M(\infty)$ a point of continuity for $f$,
the function $h_f$ is continuous at $\xi$, i.e. if
$\{\varphi_n\}\subset \Gamma$ is a sequence such that $\varphi_n
x\rightarrow \xi$ (for one and hence for any $x\in \M$), then
$h_f(\varphi_n)\rightarrow f(\xi)$.

\begin{theo} \label{Diri}Let $(M, g)$ be a closed connected rank one manifold without focal points.  Let $\nu$ be a Lyons-Sullivan measure on $\Gamma$. Then the Dirichlet problem at infinity for the random walk on $\Gamma$ generated by $\nu$ is solvable. Consequently, if $f:\ \M(\infty)\to \Bbb R$ is continuous, then $h_{f}$ is the unique $\nu$ harmonic function on $\Gamma$ extending continuously to $f$ at infinity.
\end{theo}

\begin{remark}The theorem holds true if we replace the  Lyons-Sullivan measure by any probability measure on $\Gamma$ whose support generates $\Gamma$ as a semigroup.
\end{remark}

The proof follows \cite[Theorem 4.10]{B3}. But the key lemma
Ballmann used (\cite[Lemma 4.9]{B3}) needs to be adjusted in the no
focal points  setting. We present the two parts as a whole for
completeness.

\begin{proof}[Proof of Theorem \ref{Diri}] Let  $\nu$ be a Lyons-Sullivan measure on $\Gamma$ and let $\mu$ be a $\nu$-stationary measure. Suppose  the Dirichlet problem is not solvable. Then there is a bounded measurable function $f$ and a point $\xi\in \M(\infty)$ of  continuity for $f$ such that
there is $\{\phi_n\}\subset \Gamma$ with
\begin{equation}\label{non-dirichlet}
\phi_n(x_0)\rightarrow \xi, \ \ \mbox{but}\ \
h_f(\phi_n)\not\rightarrow f(\xi).
\end{equation}

We may assume without loss of generality that $f(\xi)=0$.   Let
$\{\varphi_n\}\subset \Gamma$ be  such that $\lim_{n\rightarrow
+\infty} |h_f(\varphi_n)|$ exists and  is maximal along all such
sequences in (\ref{non-dirichlet}).  Denote by
\[
\wt{\delta}:=\lim_{n\rightarrow +\infty} |h_f(\varphi_n)|.
\]
To draw a contradiction, it suffices to show there exist $\varphi\in
\Gamma$ and $\{n_k\}\subset \Bbb N$ such that
\begin{equation}\label{e-delta}
| h_{f} (\varphi_{n_k} \varphi)  |\leq \frac{1}{2}\wt{\delta}, \ \
\mbox{for}\ \ k \ \mbox{large}.
\end{equation}
Indeed, consider the $l$-th convolution $\nu^l$ of $\nu$  defined
inductively by letting $\nu^0$ be the Dirac measure at the neutral
element of $\Gamma$ and
\[
\nu^l(\widetilde{\phi})=\sum_{\psi\in \Gamma}
\nu^{l-1}(\psi)\nu(\psi^{-1}\widetilde{\phi}), \ \ l\geq 1.
\]
It is easy to see that  $\mu$  is also  stationary with respect to
$\nu^k$ and $h_{f}$ satisfies
\begin{equation}\label{dir-e-3.2}
h_f(\varphi_{n_k})=\sum_{\psi\in \Gamma}
h_f(\varphi_{n_k}\psi)\nu^k(\psi).
\end{equation}
Let  $\nu^k(\varphi)=\alpha$.   We can break $\Gamma$ into three
subsets $\{\varphi\},  G, L$,   where  $G\subset \Gamma$ is finite
so that $L=\Gamma\backslash (G\cup \{\varphi\})$ satisfies
$\nu^k(L)\sup|f|<\alpha\dd/2$.  Then  we have by (\ref{dir-e-3.2})
that
\begin{eqnarray*}
\lim_{k\rightarrow +\infty} |h_f(\varphi_{n_k})|&<& \lim_{k\rightarrow +\infty}\big | \sum_{\psi\in G} h_f(\varphi_{n_k}\psi)\nu^k(\psi)\big|+  \lim_{k\rightarrow +\infty} |h_f(\varphi_{n_k}\varphi)|\alpha +\frac{1}{2}\alpha\dd\\
&\leq& \nu(G)\cdot\dd+\alpha\cdot \dd  \quad{\textrm{(Recall that }}\varphi_{n_k}\psi x_0 \to \xi.)\\
&\leq & \dd.
\end{eqnarray*}
This  will contradict the choice of $\dd$.

For (\ref{e-delta}),  we  firstly choose by  continuity of $f$ at
$\xi$  a number $0<\epsilon< \frac{\dd}{6\max\{\sup|f|, 1\}}$ with
\[
|f(\eta)|<\frac{1}{3}\dd, \ \mbox{for}\ \ \eta\in C_{x_0,
\xi}(\epsilon),
\]
where $C_{x_0, \xi}(\epsilon)$ is the shadow at infinity of the
cone of at $x_0$ with axis $\dot\g_{x_0, \xi}(0)$  and angle
$\epsilon$, i.e.
\[
C_{x_0, \xi}(\epsilon)=\{\eta\in \M(\infty):\ \ \angle_{x_0}(\eta,
\xi)<\epsilon\}.
\]
Next, we claim there exists a neighborhood  $U$ of  some point at
$\M(\infty)$ with $\mu(U)<\epsilon$,  $\varphi\in \Gamma$ and a
sequence $\{n_k\}\subset \Bbb N$  such that
\begin{equation}\label{cone-e-3.3}
\varphi_{n_k}\varphi(\M(\infty)\backslash U)\subset C_{x_0,
\xi}(\epsilon).
\end{equation}
With this, we have by the definition  of the function  $h_{f}$ (see
(\ref{eq-h_f})) that
\begin{eqnarray*}
|h_f(\varphi_{n_k}\varphi)|&\leq& \left|\int_{\M (\infty)\backslash U} f(\varphi_{n_k}\varphi(\xi))\  d\mu(\xi)\right|+\left|\int_{U} f(\varphi_{n_k}\varphi(\xi))\ d\mu(\xi)\right|\\
&\leq& \mu(\M (\infty)\backslash U)\cdot \frac{\dd}{3}+\mu(U)\sup|f|\\
&<& \frac{1}{2}\dd.
\end{eqnarray*}

What left is to show (\ref{cone-e-3.3}).  Let $\{n_k\}\subset \Bbb
N$ be  such that $\varphi_{n_k}^{-1}x_0\rightarrow \eta\in
\M(\infty)$ as $k$ goes to infinity. Select two points $p, q\in
\M(\infty)$ with $p$ hyperbolic and  two disjoint neighborhoods
$U^*_p, U_q$ in $\wh M$ around $p, q$, respectively,  so that their
union is apart from  $\{\eta\}\cup C_{x_0, \xi}(\epsilon)$.  We can
further require $\mu(U_q)<\epsilon$. Apply Lemma \ref{hy.angle} to
the hyperbolic point $p$. We obtain a neighborhood $U_p\subset
U^*_p$ so that
\begin{equation}\label{zero}
\sup\limits_{u\in U_p} \angle_{\varphi_{n_k}^{-1} x_0}(x_0,
u)\rightarrow  0, \ \mbox{as}\ \ k\rightarrow +\infty.
\end{equation}
For neighborhoods $U_p, U_q$ of $p, q$, we apply Lemma
\ref{hy.neigh} to obtain an isometry $\varphi\in \Gamma$ with
\[
\varphi(\wh M\backslash U_q)\subset U_p, \ \varphi^{-1}(\wh
M\backslash U_p)\subset U_q.
\]
We show $U_q$ and $\varphi$ satisfy the requirement of
(\ref{cone-e-3.3}).  Let $y\in \M(\infty)\backslash U_q$. We have
\begin{eqnarray*}
\angle_{x_0}\left(\xi, \varphi_{n_k}\varphi(y)\right) &\leq&
 \angle_{x_0}\left(\xi, \varphi_{n_k}(x_0)\right)+ \angle_{x_0}\left(\varphi_{n_k}(x_0),  \varphi_{n_k}\varphi(y)\right)\\
&=& \angle_{x_0}\left(\xi, \varphi_{n_k}(x_0)\right)+ \angle_{\varphi_{n_k}^{-1}x_0}\left(x_0, \varphi(y)\right)\\
&\leq& \angle_{x_0}\left(\xi, \varphi_{n_k}(x_0)\right)+\sup_{u\in
U_p}\angle_{\varphi_{n_k}^{-1}x_0}(x_0, u),
\end{eqnarray*}
where the first quantity goes to zero as $k$ goes to infinity since
$\varphi_{n_k} (x_0)\to \xi$ (as $k\rightarrow \infty$) and the
second quantity goes to zero by (\ref{zero}).  This shows
(\ref{cone-e-3.3}) holds true for $k$ large.
\end{proof}

\

\subsection{Coincidence of two classes of hitting measures}

To show the coincidence of the hitting probabilities at $\M(\infty)$
of the Brownian motion (starting at $x_0$) and $\nu$ random walk on
$\Gamma$, it suffices to show that  for $\overline{\P}$  almost all
trajectories $(c, \omega)$,  $c_{\omega}(t)$ tends to the same limit
as the sequence $\{\varphi_n(c, \omega)\}$.   We need two more
lemmas.

\begin{lem}\label{Jacobi}(\cite[Proposition 4]{EOs})
Let $(M, g)$ be a closed connected Riemannian manifold without focal
points. For any $x_0\in \M$, there exist positive numbers $\alpha$
and $T$ such that for $t\geq T$, the equality
\[
\|J(t)\|\geq \alpha t^{\frac{1}{2}}
\]
is satisfied uniformly by all Jacobi fields which vanish initially
at $x_0$ and have initial covariant derivative of length $1$.
\end{lem}

\begin{lem}\label{Guivarch}(see \cite[Theorem 3.14]{B2})  There exists a number $\beta>0$ such that the random sequence $\{\varphi_n\}$ in $\Gamma$  satisfies
\[
\lim\limits_{n\rightarrow \infty} \frac{1}{n} d(x_0, \varphi_n
(x_0))=\beta.
\]
\end{lem}

\

\begin{theo}\label{coin} The random Brownian path converges at $\M(\infty)$. The hitting measure at $\M(\infty)$ coincides with the hitting measure at $\M (\infty)$ of the $\nu$-random walk on $\Gamma$.
\end{theo}
\begin{proof}
First, we have for $\overline{\P}$ almost all trajectory $(c,
\omega)$ that
\begin{equation}\label{sqrt-increase}
\varlimsup\limits_{n\rightarrow \infty} \frac{1}{\sqrt{n}}
\max\limits_{T_n < t<T_{n+1}} d(c_{\omega}(t), \varphi_n(c, \omega)
x_0)=0.
\end{equation}
To see this, for any $\epsilon>0$ let
$$A_{n, \epsilon}=\{(c, \omega):\ \  \max\limits_{T_n< t<T_{n+1}} d(c_{\omega}(t), \varphi_n(c, \omega) x_0)>\epsilon \sqrt{n}\}.$$  Then  by  iii) of Proposition \ref{Lyons-S}, we have\[
\sum\limits_{n=0}^{+\infty} \overline{\P}(A_{n, \epsilon})\leq
\sum\limits_{n=0}^{+\infty} \delta^{\epsilon \sqrt{n}},
\]
which is finite since $\sum_{n=0}^{+\infty} \delta^{\epsilon
\sqrt{n}}\leq 1+2\sum_{l=1}^{+\infty}(l+1)\delta^{\epsilon
l}<+\infty$. So we have by Borel-Cantelli lemma that $\cap_{m\in
\Bbb N}\cup_{n\geq m}A_{n, \epsilon}$ has $\overline{\P}$
probability $0$ and (\ref{sqrt-increase}) follows since $\epsilon>0$
is arbitrary.

By Lemma \ref{Guivarch}, we have for $\overline{\P}$ almost all $(c,
\omega)\in W\times \Omega$,
\[
\varliminf_{n\rightarrow \infty}\frac{1}{n} d(x_0, \varphi_n(c,
\omega) x_0)=\beta>0
\]
and hence we have by (\ref{sqrt-increase}) that $d(x_0,
c_{\omega}(t))\rightarrow \infty$ as $t\rightarrow \infty$ as well.

Finally, Theorem \ref{coin} directly follows from
\begin{equation}\label{angle}
\max\limits_{T_n< t<T_{n+1}}  \angle_{x_0}(c_{\omega}(t),
\varphi_n(c, \omega) x_0)\rightarrow 0, \ \mbox{as}\ \  n\rightarrow
\infty.
\end{equation}
Suppose (\ref{angle}) doesn't hold $\overline{\P}$  almost
everywhere. For a set of $\overline{\P}$ positive measure of  $(c,
\omega)$, there is a $\theta>0$ and infinitely many $n$ such that
\begin{equation}\label{eq-3.6}
\max\limits_{T_n<  t<T_{n+1}}  \angle_{x_0}(c_{\omega}(t),
\varphi_n(c, \omega) x_0)>\theta.
\end{equation}
For any $0<\epsilon<\frac{1}{8}\beta$,  for $\overline{\P}$ almost
all $(c, \omega)$, there is  $N>\max\{4T/\beta,  T\}$ (where $T$ is
from Lemma \ref{Jacobi}) such that for  $n>N$,  we have
\begin{eqnarray}\label{eq-3.7}
d_1(n)&=&  \max\limits_{T_n< t<T_{n+1}}  d(c_{\omega}(t), \varphi_n(c, \omega) x_0)<\epsilon \sqrt{n},\\
d_2(n)&=&  d(x_0, \varphi_n(c, \omega) (x_0))>\frac{1}{2}\beta
n.\label{eq-3.8}
\end{eqnarray}
Take $(c, \omega)$ so that (\ref{eq-3.6}), (\ref{eq-3.7}) and
(\ref{eq-3.8}) hold true. For $n>N$ and $t\in (T_n, T_{n+1})$,
consider the geodesics $\gamma, \wt\g$ which start at $x_0$ and
point at $c_{\omega}(t)$ and $\varphi_n(c, \omega) x_0$,
respectively.  On the one hand,  we have by (\ref{eq-3.7}) and the
triangle inequality  that
\begin{equation}\label{e-contra}
d(c_{\omega}(t), \varphi_n(c, \omega) x_0)\geq d(\g(d_2(n)), \wt\g
(d_2(n)))-\epsilon \sqrt{n}.
\end{equation}
On the other hand, since $d_2(n)>\beta n/2>T$, we have by Lemma
\ref{Jacobi} and (\ref{eq-3.6}) that
\begin{eqnarray*}
d(\g (d_2(n)), \wt\g (d_2(n)))&>& \theta \alpha d_2(n)^{\frac{1}{2}}\\
&>& \theta\alpha (\frac{1}{2}\beta n)^{\frac{1}{2}}\\
&>& 2\epsilon \sqrt{n},
\end{eqnarray*}
if we choose  $\epsilon<2^{-1} \theta \alpha
(\frac{1}{2}\beta)^{\frac{1}{2}}$. This, together with
(\ref{e-contra}),  contradicts  (\ref{eq-3.7}). \end{proof}

\

\subsection{Proof of Theorem \ref{har.mea}}

\begin{proof}
Let $m$ be a $\W$ harmonic probability measure on $SM$. Then, there
is a unique $\Gamma$-invariant measure $\wt m$ on $S\M$ which
coincides with $m$ locally. Seen as a measure on $\M \times \MM$, we
claim that $\wt m$ is given, for any continuous $f$ with compact
support, by:
\begin{equation}\label{harm.meas} \int f(x, \xi )\  d\wt m (x,\xi) \; = \; \int _{\M} \left( \int _{\MM} f(x,\xi ) \ d m _x (\xi) \right)\ dx, \end{equation}
where $m_x$ is the hitting measure of Brownian motion at
$\M(\infty)$ starting at $x$ and  $dx$ is proportional to the
Riemannian volume on $\M$.

Firstly,  there is  a family of probability measures $x\mapsto m_x$
on $\MM$  with (\ref{harm.meas}) holds  such that, for all $g $
continuous on $\MM$, $x\mapsto \int g (\xi )\ d m_x (\xi)  $ is a
harmonic function on $\M$.  This follows from \cite{Ga}. On the one
hand, the measure $\wt m$ projects on $\M$ as a $\Gamma$-invariant
measure satisfying $\int \D f \ dm  = 0 $. The projection of $\wt m
$ on $\M$ is proportional to Volume and  formula (\ref{harm.meas})
is the desintegration formula. On the other hand, if one projects
$\wt m$ first on $\MM$, there is a probability measure ${\bf m}$ on
$\MM$ such that $$\int f(x,\xi) \ d\wt m (x,\xi) \;= \int_{\MM}
\left( \int_{\M} f(x,\xi) \ dm_\xi (dx) \right)\  d{\bf m} (\xi).$$
For ${\bf m} $-a.e. $\xi $, the measure $m_\xi $ is a harmonic
measure on $\M$; therefore, for ${\bf m}$-a.e. $\xi $,  there is a
positive harmonic function $k_\xi (x) $ such that $m_\xi = k_\xi (x)
{\textrm{Vol }}$.  Comparing the two expressions for $\int f \ d\wt
m$, we see that, up to a normalizing constant, the measure $m_x $ is
given by $$ m _x (d\xi) \; = \; k_\xi (x) {\bf m}(d\xi).$$ We
normalize ${\bf m}$ in such a way that $k_\xi(x_0)=1$ for almost all
$\xi$. Then,  $x \mapsto \int _{\MM}  g (\xi )\  d m_x (\xi ) $ is
indeed a harmonic function.

Next,  for any  $x_0 \in \M$, let $\nu$ be the corresponding
Lyons-Sullivan measure on $\Gamma$.  Since for all $g$ continuous on
$\M (\infty)$, $x\mapsto m_x(g)$ is a harmonic function and that
$m_{\varphi x_0}=\varphi_{*}m_{x_0}$ for $\varphi\in \Gamma$, it
follows that the measure $m_{x_0}$ is a stationary measure for
$\nu$, i.e.
\[
m_{x_0}=\sum_{\varphi\in \Gamma}\varphi_{*} m_{x_0} \nu(\varphi).
\]
So we conclude from Theorem \ref{sta.mea} and Theorem \ref{coin}
that $m_{x_0}$ is the hitting probability at $\M (\infty)$ of the
Brownian motion starting at $x_0$. Since $x_0$ was arbitrary in the
above reasoning, we have the desired expression of the lift of $m$
as in (\ref{harm.meas}).

Finally, we have by the solvability of  Dirichlet problem (Theorem
\ref{Diri}) that  each $m_{x}$ is fully supported on $\M(\infty)$.
It follows that the unique $\W$ harmonic measure $m$ is fully
supported on $SM$.
\end{proof}

\section{A linear drift characterization of local symmetry}

For Theorem \ref{rankone-thm}, it remains  to show Proposition
\ref{as.har}.  Consider the action of $\Gamma$ on
$\widehat{M}=\M\cup \M (\infty)$.  Let  $X_{M}$ be  the quotient of
the space $\M\times \widehat M$ by the diagonal action of $\Gamma$.
To each $\xi\in \widehat M$ is associated the projection of
$\widehat W_{\xi}$ of $\M\times \{\xi\}$ in $X_{M}$.   As a subgroup
of $\Gamma$, the stabilizer $\Gamma_{\xi}$ of the point $\xi$ acts
discretely on $\M$ and the space $\widehat W_{\xi}$ is homeomorphic
to the quotient of $\M$ by $\Gamma_{\xi}$. Put on each $\widehat
W_{\xi}$ the smooth structure and the metric inherited from $\M$.
The manifold $\widehat W_{\xi}$ and its metric vary continuously on
$X_{M}$. The collection of all $\widehat W_{\xi}, \xi\in \widehat M$
form a continuous lamination $\widehat\W$ of $X_{M}$ with leaves
which are manifolds locally modeled on $\M$.  Denote by
$\Delta^{\widehat\W}$ the laminated Laplace operator acting on
functions which are smooth along the leaves of the lamination.  A
Borel  measure $m$  on $X_{M}$ is called harmonic (with respect to
$\widehat \W$) if it satisfies, for all $f$ for which it makes
sense,
\[
\int \Delta^{\widehat \W} f\ dm=0.
\]
Let $\widehat m$  be  the $\Gamma$ invariant extension of a harmonic
measure  $m$ on $\M\times \widehat M$.  There exists a finite
measure $\widehat {\bf m}$ on $\widehat M$ (\cite{Ga}) and, for
$\widehat{\bf m}$-almost every $\xi$, a positive harmonic function
$k_{\xi}(x)$ with $k_{\xi}(x_0)=1$ such that the measure $m$ can be
written as
 \[
 \widehat m=k_{\xi}(x)(dx\times \widehat {\bf m} (d \xi)).
 \]
 The set of harmonic probability measures is a weak* compact nonempty set of measures on $X_{M}$ (\cite{Ga}).  A harmonic probability measure $m$ is called ergodic if it is extremal among harmonic probability measures.   As a corollary of the results in \cite{L4} and \cite[Proposition 4.2]{Zi2}, we have the following proposition.

\begin{prop}\label{har.busemann} Let $(M, g)$ be a closed connected Riemannian manifold  without focal points and $\ell^2=h$. Then, there exists an ergodic harmonic probability measure $m$ on $X_M$ such that $\Delta  b_{\xi}=\ell$ for $\widehat{\bf m}$ almost all $\xi\in  \widehat{M}$ (see below for definitions).\end{prop}

Indeed,  let $(M, g)$ be a closed connected Riemannian manifold.
One can consider the Busemann compactification of $\M$ as follows.
Fix a point $x_0\in \M$ as a reference point.  For each point $x\in
\M$, define a function $ b_{x}(z)$ on $\M$ by
\[
 b_{x}(z):= d(x, z)-d(x, x_0),\ \forall z\in \M.
\]
The assignment $x\mapsto  b_x$ is continuous, one-to-one and takes
values in a relatively compact set of functions for the topology of
uniform convergence on compact subsets of $\M$. The Busemann
compactification $ \breve M$ of $\M$ is the closure of $\M$ for that
topology.  The space $\breve M$ is a compact separable space.   The
Busemann boundary $\breve M (\infty):=\breve M\backslash \M$ is
compact  (\cite[Proposition 1]{LW}) and is  made of $1$-Lipschitz
continuous functions  $\breve \xi$ on $\M$ such that $\breve \xi
(x_0)=0$.     It is shown in \cite{L4} that  if $\ell^2=h$, then
there exists an ergodic harmonic probability measure $\breve m$
corresponding to the quotient  of the space $\M\times \breve M$ by
the diagonal action of  $\Gamma$  such that $\Delta \breve \xi=\ell$
for $\breve{\bf m}$ almost all $\breve \xi\in  \breve{M}$. (Here, by
$\Delta$, we mean the  Laplacian in the distribution sense.)  Note
that in case $M$ has no focal points, there exists  a homemorphism
$\breve \pi: \breve M\mapsto \M $ (\cite[Proposition 4.2]{Zi2})
which satisfies $\breve \xi=b_{\breve \pi \breve \xi}$ for $\breve
\xi\in \breve M$. Proposition \ref{har.busemann}  follows
immediately  by letting $m$ be the projection of $\breve m$ on
$X_{M}$.

\begin{proof}[Proof of Proposition \ref{as.har}] Let $(M, g)$ be a closed connected rank one Riemannian manifold without focal points.  Assume $\ell^2=h$.  One can obtain a  harmonic measure  $m$ satisfying Proposition \ref{har.busemann}  by describing its $\Gamma$ invariant extension $\widehat{m}$ (\cite[p. 720]{L4}).   Set
\begin{equation}\label{eq-m_t}
\widehat m_t:=\int_{\M}  p (t, x, y) dy \ \frac{dx}{\mbox{Vol} M},
\end{equation}
where $p(t, x, y),  t \in \R_+, x,y \in \M $  is the heat kernel on
$\M$.  Then pick up $\widehat m$ as any ergodic decomposition of a
limit point of $\frac{1}{T}\int_{0}^{T} \widehat m_t\ dt$. Its
corresponding harmonic measure $m$ on $X_{M}$ satisfies Proposition
\ref{har.busemann} as required.

Note that the random Brownian path converges at $\M(\infty)$ by
Theorem \ref{coin}.  Hence we conclude from (\ref{eq-m_t}) that
$\widehat m$ is actually pushed to be supported on $\M\times \M
(\infty)$. Thus,  $m$ is one (and hence is the only one by Theorem
\ref{har.mea}) harmonic measure for the stable foliation.   We also
conclude from Theorem \ref{har.mea} that  $m$ is fully supported on
$SM$.

Now  we have $\widehat{\bf m}$ is fully supported on $\M (\infty)$.
Moreover, by Proposition  \ref{har.busemann}, we have for
$\widehat{\bf m}$-a.e. $\xi$,  $\Delta b_{\xi} =\ell$.   Recall from
section \ref{sec-Busemann} that the map $x\mapsto b_{\xi}(x)$ is of
class $C^2$ and $\Delta_x b_\xi$ depends continuously on $\xi\in \M
(\infty)$.  We conclude  that  the Laplacian of the Busemann
function $B$ on $SM$ is constant $\ell$.  This shows   $(\M,
\wt{g})$ is asymptotically harmonic.
\end{proof}

\begin{remark} One can  also show that any $m$ in Proposition \ref{har.busemann}  is fully supported on $SM$ without using the explicit description of the unique harmonic measure for the stable foliation  in Theorem \ref{har.mea}.   Yet, we prefer to use it since it  is independent of the assumption of $\ell^2=h$, has its own interest in Dirichlet problem and might be useful in a further study of Martin boundary in no focal points case.  Here is one approach suggested by Zimmer.
Any such $m$  also satisfies (\cite{L4})
\begin{equation*}\label{equ-led-grad}
\nabla_x \ln k_{\xi}(x)=-\ell \nabla_x b_{\xi}, \ \mbox{for} \
\widehat m\mbox{-a.e.} \ (x, \xi)\in \M\times \widehat M.
\end{equation*}
So $\widehat{\bf m}$-a.e.  $\xi$ must belong to $\M(\infty)$ because
of  the non-differentiability of $b_{\xi}$ at $\xi\in \M$ (cf.
\cite[Theorem 5.2]{Zi2}). Thus, $\widehat  m$ is supported on
$\M\times \M (\infty)$.  By using  the density of hyperbolic points
in $\M (\infty)$ (Lemma \ref{hy.point}),  Lemma  \ref{hy.neigh}  and
a similar argument as in \cite[Lemma 4.1]{Kn2},  one can obtain that
each $\widehat  m_x, x\in \M$,  of the disintegration of $\widehat
m$ is fully supported on $\M(\infty)$.  It follows that   $ m$ is
fully supported on $SM$. \end{remark}

\

\small{{\bf{Acknowledgments}} The first author was partially
supported by NSF: DMS 0811127. The second author was partially
supported by NSFC: 10901007 and China Scholarship Council; she would
also like to thank the department of mathematics of the university
of Notre Dame for hospitality during her stay. }


\small\begin{thebibliography}{99}

 \bibitem[{\bf AS}]{AS} M. Anderson and R. Schoen, Positive harmonic functions on compact manifolds of negative curvature, \emph{Ann. Math.}  {\bf 121} (1985), 429--461.


%\bibitem[{\bf B1}]{B1} W. Ballmann, Nonpositively curved manifolds of higher rank, {\em Ann. Math}, {\bf 122} (1985), 597--609.

\bibitem[{\bf B2}]{B2} W. Ballmann, On the Dirichlet problem at infinity for manifolds of nonpositive curvature, {\em Forum Mathematicum}, {\bf 1} (1989), 201--213.

\bibitem[{\bf B3}]{B3} W. Ballmann, Lectures on spaces of nonpositive curvature, {\em DMV Seminar},  {\bf 25} (1995).


\bibitem[{\bf BE}]{BE}  W. Ballmann and P. Eberlein,   Fundamental groups of manifolds of nonpositive curvature, \emph{J. Differential Geom.}   {\bf 25}  (1987), no. 1, 1--22.

\bibitem[\bf BCG1]{BeCG91} G. Besson, G. Courtois and  S. Gallot, Volume et entropie minimale des espaces sym\'etriques, \emph{Invent. math.}  {\bf 103} (1991), 417--445.



\bibitem[{\bf  BCG2}]{BCG95} G. Besson, G. Courtois and S. Gallot, Entropies et rigidit\'es des espaces localement sym\'etriques de courbure strictement n\'egative, {\em Geom. Func. Anal.}   {\bf 5} (1995), 731--799.

\bibitem[{\bf BFL}]{BFL} Y. Benoist, P. Foulon and F. Labourie,  Flots d'Anosov \`a distributions stables et instables diff\'erentiables, {\em  J. Amer. Math. Soc. }  {\bf 5} (1992), 33--74.

%\bibitem[\bf Br]{Br81} R. Brooks,
%The fundamental group and the spectrum of the Laplacian.
%\emph{Comment. Math. Helv.}  {\bf 56} (1981), no. 4, 581--598.



%\bibitem[{\bf BS}]{BS} K. Burns and R. Spatzier, Manifolds of nonpositive curvatuer and their buildings, {\em Publications math. IHES}, {\bf 65} (1987),  35--59.


% \bibitem[{\bf CFL}]{CFL} J. G.  Cao,  H. J. Fan  and  F. Ledrappier,  Martin points on open manifolds of non-positive curvature.
%\emph{Trans. Amer. Math. Soc.}  {\bf 359}  (2007), no. 12,
%5697--5723.



%\bibitem[{\bf E1}]{E1} P. Eberlein, Geodesic flows on negatively curved manifolds, II, {\em Transactions Amer. math. Soc.}, {\bf 178} (1973), 57--82.

\bibitem[\bf EO]{EO} P. Eberlein and  B. O'Neill, Visibility manifolds, \emph{Pacific J. Math.}  {\bf 46}  (1973), 45--109.




\bibitem[\bf E]{E}  J.-H. Eschenburg,  Horospheres and the stable part of the geodesic flow, \emph{Math. Z. }   {\bf 153} (1977), no. 3, 237--251.

\bibitem[\bf EOs]{EOs} J.-H.  Eschenburg and J.-J.  O'Sullivan,
Growth of Jacobi fields and divergence of geodesics, \emph{Math. Z.}
{\bf 150} (1976), no. 3, 221--237.

\bibitem[{\bf  FL}]{FL}   P. Foulon and F. Labourie,  Sur les vari\'et\'es compactes asymptotiquement harmoniques, {\em  Invent. Math. }  {\bf 109} (1992), 97--111.


\bibitem[{\bf FM}]{FM} A.  Freire and  R. Ma\~{n}\'{e},  On the entropy of the geodesic flow in manifolds without conjugate points,  \emph{Invent. Math.}  {\bf 69} (1982), no. 3, 375--392.



\bibitem[{\bf Fu}]{Fu}  H. Furstenberg,
Boundary theory and stochastic processes on homogeneous spaces.
\emph{Harmonic analysis on homogeneous spaces } (Proc. Sympos. Pure
Math., Vol. XXVI, Williams Coll., Williamstown, Mass., 1972),
193--229. Amer. Math. Soc., Providence, R.I., 1973.


\bibitem[{\bf Ga}]{Ga} L. Garnett,   Foliations, the ergodic theorem and Brownian motion, {\em J. Funct. Anal. }  {\bf 51}  (1983), 285--311.

%\bibitem[{\bf Gr}]{Gr} A. Grygor'yan, Heat kernels on weighted manifolds and applications, \emph{in} The ubiquitous  heat  kernel, \emph{Contemp. Math.}  {\bf 398}  (2006),  93--191.

\bibitem[{\bf Go}]{Go}  M.-S. Goto, The cone topology on a manifold without focal points,  \emph{J. Differential Geom. }  {\bf 14}  (1979), no. 4,   595--598.


\bibitem[\bf Gu] {Gu}Y. Guivarc'h, Sur la loi des grands nombres et le rayon spectral
d'une marche al\'{e}atoire, \emph{Ast\'{e}risque,} {\bf 74} (1980),
47--98.

\bibitem[\bf Gul]{Gul} R. Gulliver,  On the variety of manifolds without conjugate points, \emph{Trans. Amer. Math. Soc.} {\bf 210} (1975), 185--201.

\bibitem[\bf He]{He} S. Helgason, {\em Differential geometry and symmetric spaces. Pure and Applied Mathematics, Vol. XII},  Academic Press, New York-London 1962.

%\bibitem[{\bf HI}]{HI}   E. Heintze and H.-C. Im Hof,  Geometry of horospheres,  {\em  J. Diff. Geom.} {\bf 12} (1977), 481--491.


\bibitem[{\bf Hu}]{Hu}  D. Hurley,
Ergodicity of the geodesic flow on rank one manifolds without focal
points, \emph{Proc. Roy. Irish Acad. Sect. A}   {\bf 86} (1986), no.
1, 19--30.

 \bibitem [\bf K1]{K1}V.-A. Kaimanovich, Brownian motion and harmonic
functions on covering manifolds. An entropic approach, {\em Soviet
Math. Dokl.} {\bf 33} (1986),  812--816.


\bibitem[{\bf Kar}]{Kar}  F.-I.  Karpelevich, The geometry of geodesics and
the eigenfunctions of the Beltrami Laplace operator on symmetric
spaces, {\em Trans. Moscow Math. Soc. 1965},  Amer. Math. Soc,
Providence (1967), 51--199.

\bibitem[{\bf Ka}]{Ka} A. Katok, Four applications of conformal equivalence to geometry and dynamics, {\em Ergod. Th. \& Dynam. Sys.} {\bf 8} (1988), 139--152.


%\bibitem[{\bf K2}]{K2} V.A. Kaimanovich,  Brownian motion on foliations: Entropy, invariant measures, mixing, {\em  Funct. Anal. Appl.} {\bf 22} (1989), 326--328.

\bibitem[{\bf Kn1}]{Kn2} G. Knieper,  On the asymptotic geometry of nonpositively curved manifods, {\em Geom. Func. Anal.}  {\bf 7} (1997), 755--782.


\bibitem[{\bf Kn2}]{Kn} G. Knieper, New results on noncompact harmonic manifolds, {\em Comment.
Math. Helv.} {\bf 87} (2012), 669--703.



\bibitem[{\bf L1}]{L1} F. Ledrappier, Harmonic measures and Bowen-Margulis measures, {\em Israel J. Math.}  {\bf71} (1990), 275--287.
\bibitem[{\bf L2}]{L2} F. Ledrappier,  A heat kernel characterization of asymptotic harmonicity, \emph{ Proc. Amer. Math. Soc.}  {\bf 118}  (1993),  no. 3, 1001--1004.


\bibitem[{\bf L3}]{L3} F. Ledrappier, Applications of dynamics to compact manifolds of negative
curvature, {\it in}  {\sl Proceedings of the ICM Z\"urich 1994},
Birkh\"auser (1995), 1195-1202.


\bibitem[{\bf L4}]{L4} F. Ledrappier, Linear drift and entropy for regular covers, {\em Geom. Func. Anal.} {\bf 20} (2010), 710--725.

\bibitem[{\bf LW}]{LW}  F. Ledrappier and X.-D. Wang,  An integral formula for the volume entropy with applications to rigidity,
\emph{J. Differential Geom. }  {\bf85} (2010), no. 3, 461--477.


\bibitem[{\bf LS}]{LS} T. Lyons and D. Sullivan, Function theory, random paths and covering spaces, {\em J. Differential Geom.} {\bf 19} (1984), 299--323.



%\bibitem[\bf O]{O}  M. A. Olshanectsky,   The Martin boundary for the Laplace-Beltrami operator on a Riemannian symmetric space of nonpositive curvature,  \emph{Uspehi Mat. Nauk} 24 1969 no. 6  {\bf150}, 189--190.



\bibitem[\bf O]{O} J.-J. O'Sullivan,  Riemannian manifolds without focal points, \emph{J. Differential Geom.}  {\bf 11} (1976), no. 3, 321--333.

\bibitem[\bf W]{Wa} J. Watkins, The higher rank rigidity theorem for manifolds with no focal points,  \emph{ http://arxiv.org/abs/1111.5428}.


\bibitem[\bf Z1]{Zi1} A. Zimmer, Asymptotically harmonic manifolds without focal points,   \emph{http://arxiv.org/abs/1109.2481}.
\bibitem[\bf Z2]{Zi2} A. Zimmer, Compact asymptotically harmonic manifolds,
\emph{http://arxiv.org/abs/1205.2271}.


\end{thebibliography}
\end{document}